\documentclass[letterpaper, 10 pt, conference]{ieeeconf}  
\IEEEoverridecommandlockouts                              
\usepackage{graphicx}
\usepackage{mathptmx} 
\usepackage{times} 
\usepackage{amsmath} 
\usepackage{amssymb}  

\usepackage{amsthm}
\usepackage{float}
\usepackage[dvipsnames]{xcolor}

\theoremstyle{plain}
\newtheorem{theorem}{Theorem}[section]
\newtheorem{proposition}[theorem]{Proposition}

\theoremstyle{definition}

\newtheorem{assumption}{Assumption}[section]
\theoremstyle{remark}
\newtheorem{remark}[theorem]{Remark}

\title{\LARGE \bf
Observer design for classes of nonlinear port-Hamiltonian systems
}

\author{Filippo Ugolini$^{1}$, Ning Liu$^{2}$, Yongxin Wu$^{2}$, Yann Le Gorrec$^{2}$, Alessandro Macchelli$^{1}$
\thanks{*This work is supported by the EIPHI Graduate School (contract:
ANR-17-EURE-0002) and Région Franche-Comté.}
\thanks{$^{1}$Filippo Ugolini and Alessandro Macchelli are with the Department of Electrical, Electronic and Information Engineering (DEI), University of Bologna, Italy.}
\thanks{$^{2}$Ning Liu, Yongxin Wu, Yann Le Gorrec are with Université Marie et Louis Pasteur, SUPMICROTECH, Institut FEMTO-ST, Besançon, France
        {\tt\small yongxin.wu@femto-st.fr}}%
}
\begin{document}

\maketitle
\thispagestyle{empty}
\pagestyle{empty}

\begin{abstract}
 
This paper presents a systematic observer design methodology for a class of port-Hamiltonian (pH) systems with state-dependent input matrices. Such systems can model a wide range of electromechanical systems, including magnetic levitation systems, MEMS devices, and electro-active polymer actuators such as DEA actuators, HASEL actuators, etc. In these applications, state-dependent input matrices naturally arise when the system is modeled under quasi-static electrical assumptions. An LPV polytopic embedding framework, together with LMI-based synthesis conditions, is proposed. The nonlinear error dynamics are represented as a convex combination of linear vertex systems using an integral mean value representation, which enables systematic computation of the observer gains that ensures exponential convergence. Both constant-gain and gain-scheduled observers are derived. Numerical results demonstrate the effectiveness of the proposed observer, with the gain-scheduled design achieving a significant increase in the maximum certifiable decay rate compared with constant-gain approaches, thereby reducing conservatism.

\end{abstract}

\section{Introduction}
Electromechanical systems such as dielectric elastomer actuators (DEAs) \cite{Liu2022, Hammoud2026}, HASEL actuators \cite{Cisneros2025, Yeh_HASEL_2022}, magnetic levitation systems, and micro-electromechanical (MEMS) devices \cite{L2gain_book} exhibit strong nonlinearities and multiphysics coupling. In many of these systems, state-dependent input matrices naturally arise from the underlying electromechanical interactions.
Port-Hamiltonian (pH) systems provide an energy-based modeling framework that captures storage, dissipation, and interconnection structures in a unified representation for multi physical domains \cite{vanderSchaft2000}. This structure ensures passivity by construction and offers a natural foundation for nonlinear control design.

Many control techniques for pH systems, including interconnection and damping assignment passivity-based control (IDA-PBC) \cite{OrtegaSpongMaschkeAllgower02}, and control by interconnection \cite{OrtegaCbI08}, rely on state feedback. However, in practice, only partial state measurements are available. This limitation necessitates the development of state observers that are compatible with the port-Hamiltonian structure.
With this in mind, this paper focuses on observer design for pH systems in an open-loop configuration, providing a foundation for future closed-loop implementations. 

Observer design for nonlinear pH systems has been extensively studied: passivity-based observers \cite{Venkatraman2010} guarantee exponential convergence for specific system classes, while subsequent works have proposed alternative designs exploiting damping properties \cite{Pfeifer2021} or structural decompositions \cite{Granados2024, Rojas2021}. More recently, contraction-based approaches \cite{Spirito2024} have provided strong convergence guarantees while preserving the pH structure. However, most existing results rely on restrictive assumptions, notably constant 
input matrices. This limitation is critical in electromechanical systems, where the input mapping is inherently state-dependent. This paper addresses this issue by proposing an observer design methodology for nonlinear pH systems with state-dependent input matrices. The nonlinear estimation error dynamics are reformulated using an exact integral representation based on the mean value theorem, leading to a linear parameter-varying (LPV) representation. A polytopic embedding is then employed to derive Linear Matrix Inequality (LMI) conditions ensuring exponential convergence of the error dynamics.

The main contributions are: (i) an exact integral reformulation of the nonlinear error dynamics for pH systems with state-dependent input matrices, (ii) a polytopic LPV embedding enabling tractable observer synthesis, (iii) LMI-based conditions guaranteeing exponential convergence, and (iv) a comparison between constant-gain and gain-scheduled observers highlighting conservatism reduction.

The remainder of this paper is organized as follows. Section \ref{sec:PRELIMINARIES} introduces the system model and problem formulation. Section \ref{sec:Observer design} presents the observer design methodology. Section \ref{sec:Simulation} provides numerical validation. Conclusions and perspectives are presented in Section \ref{sec:conclusions}.

\section{Preliminaries and problem formulation}\label{sec:PRELIMINARIES}

\subsection{Port-Hamiltonian Systems}\label{subsec:Mech_pH}

Port-Hamiltonian systems constitute a unified modeling framework for multi physical systems, based on the concepts of energy storage, dissipation, and power-conserving interconnections \cite{vanderSchaft2000}. In this paper, we present the following class of pH systems. Given the state vector $\mathbb{R}^{2n}\ni x = [q^T, \; p^T]^\top$, where $q \in \mathbb{R}^n$ denotes the unactuated states and $p \in \mathbb{R}^n$ the actuated states, the system dynamics is formulated as:
\begin{equation}
\dot{x} = (J - R)\nabla H(x) + g(x) u(t), \quad y = g(x)^\top \nabla H(x),
\label{eq:plant}
\end{equation}
where both the interconnection ($J = -J^\top\in \mathbb{R}^{2n\times 2n}$) and dissipation ($\mathbb{R}^{2n\times 2n}\ni R = R^\top \geq 0$) matrices are constant:
\begin{equation}
J = \begin{bmatrix} 0 & I_n \\ -I_n & 0 \end{bmatrix}, \quad
R = \begin{bmatrix} 0 & 0 \\ 0 & \eta \end{bmatrix},
\end{equation}
with $\mathbb{R}^{n\times n}\ni\eta > 0$ the damping coefficient. $u\in\mathbb{R}^m$ and $y\in\mathbb{R}^m$ are power conjugated input and output.

The Hamiltonian function is quadratic in the state:
\begin{equation}
H(x) = \frac{1}{2}q^\top K q+\frac{1}{2} p^\top M^{-1} p   = \frac{1}{2} x^\top Q x,
\end{equation}
where $K=K^\top$ and $M=M^\top$ are positive definite, which are related to the energy matrices. Define $Q=\text{diag}\left[K,\; M^{-1}\right]$ as the Hessian of the Hamiltonian.
In this paper, we consider a specific input matrix that 
depends on the state variables $q$:
\begin{equation}
g(x) = \begin{bmatrix} 0 & g_2(q) \end{bmatrix}^{\top}, 
\label{eq:input_map}
\end{equation}
with $g_2(q)\in\mathbb{R}^{n\times m}$.

\begin{remark}
    The system \eqref{eq:plant} stems from the modeling of electromechanical systems like DEAs in \cite{Hammoud2024} under the assumption that the electric dynamics is fast enough to be considered quasi-static.
\end{remark}
\subsection{Observer Structure}
Assume the system \eqref{eq:plant} is fully observable and consider the Luengerger-like structure with constant $L$ adapted to system \eqref{eq:plant}:
\begin{equation}
\dot{\hat{x}} = (J - R)\nabla H(\hat{x}) + g(\hat{q}) u(t) + L (y - \hat{y}),
\label{eq:observer}
\end{equation}
where $\hat{x} = [\hat{q}^\top, \; \hat{p}^\top]^\top$ is the estimated state, $\hat{y} = g(\hat{q})^\top \nabla H(\hat{x})$, and $L \in \mathbb{R}^{2n \times m}$ is the observer gain to be designed.
Define the estimation error as $\tilde{x} := x - \hat{x} = [\tilde{q}^\top, \; \tilde{p}^\top]^\top$. Subtracting \eqref{eq:observer} from \eqref{eq:plant} yields
\begin{equation}
\begin{aligned}
\dot{\tilde{x}} &= (J - R)Q\tilde{x} + \big(g(q) - g(\hat{q})\big)u(t) + L\big(y - \hat{y}\big) \\
&= A_0 \tilde{x} + \gamma(x,u) - \gamma(\hat{x},u),
\end{aligned}
\label{eq:error_dynamics}
\end{equation}
where $A_0 := (J - R)Q$, and $\gamma$ is defined as
\begin{equation}
\gamma(x,u) := g(x)u - L g(x)^\top Q x.
\label{eq:gamma}
\end{equation}
The representation in \eqref{eq:error_dynamics} isolates the nonlinear terms and expresses the error dynamics as a linear term in $\tilde{x}$ plus a state-dependent component. It is suitable for the contraction-based procedure considering a quadratic Lyapunov function.

\subsection{Structural Assumptions}

The following assumptions on the Hamiltonian function and system structure are considered, which are critical for the convergence analysis presented in Section \ref{sec:Observer design}.
\begin{assumption}[{Bounded Hessian \cite[\emph{Assumption~I}]{Spirito2024}}]\label{ass:bounded_hess}
    The Hessian of the Hamiltonian function is uniformly bounded and positive definite on the operating domain $\mathcal{D}$. That is, there exist positive constants $h_1, h_2 > 0$ such that for all $x \in \mathcal{D}$:
\begin{equation}
h_1 I \leq \nabla^2 H(x) \leq h_2 I.
\label{eq:bounded_hessian}
\end{equation}
\end{assumption} 
\begin{assumption}[{Local Lipschitz Continuity \cite[\emph{Lemma 5}]{Ichalal2012}}]\label{ass:continuity}
The gradient $\nabla H(x)$ is locally Lipschitz continuous on $\mathcal{D}$. For systems with state-dependent input matrix $g(x)$, we additionally require that $g(x)$ and its derivative $\frac{\partial g}{\partial x}$ are locally Lipschitz continuous on $\mathcal{D}$.
\end{assumption}
These regularity conditions ensure the existence and uniqueness of solutions to both the plant dynamics \eqref{eq:plant} and the observer dynamics \eqref{eq:observer}, as well as the validity of the integral mean value representation developed in the next Section.

\section{Observer Design via Polytopic Embedding}\label{sec:Observer design}
\subsection{Integral Representation of Nonlinear Terms}
The main challenge in analyzing the convergence of the error dynamics \eqref{eq:error_dynamics} is handling the nonlinear term $\gamma(x,u) - \gamma(\hat{x},u)$. To address this issue, we propose the following integral mean value representation to separate the error state $\tilde{x}$ from the nonlinear term for observer design.
\begin{proposition}\label{prop:integral_rep}
Let $\gamma(\cdot, u) : \mathbb{R}^{2n}\times \mathbb{R}^m \to \mathbb{R}^{2n}$ be continuously differentiable in $x$ for each fixed $u \in \mathbb{R}^m$. Then the nonlinear term admits the exact representation:
\begin{equation}
\gamma(x,u) - \gamma(\hat{x},u) = \left( \int_0^1 \frac{\partial \gamma}{\partial x}(\bar{x}(s), u) \, ds \right) \tilde{x},
\label{eq:integral_rep}
\end{equation}
where $\bar{x}(s) = \hat{x} + s\tilde{x}$ for $s \in [0,1]$ is the straight-line parametrization connecting $\hat{x}$ and $x$.
\end{proposition}
\begin{proof}
The proof follows \cite{Ichalal2012}. Define the scalar-parameterized function $\Phi(s) := \gamma(\bar{x}(s),u)$ where $\bar{x}(s) = \hat{x} + s(x - \hat{x})$ for $s \in [0,1]$ (see \cite[Theorem V.1]{Spirito2024}). By construction, $\bar{x}(0) = \hat{x}$ and $\bar{x}(1) = x$. Following Assumption~\ref{ass:continuity}, the mean value theorem in \cite[Lemma 4]{Ichalal2012} and the chain rule \cite[Theorem 1.16]{Faris2016}, the following representation is obtained:
\begin{equation}
\frac{d\Phi}{ds} = \frac{\partial \gamma}{\partial x}(\bar{x}(s),u) \frac{d\bar{x}}{ds} = \frac{\partial \gamma}{\partial x}(\bar{x}(s),u)\tilde{x}.
\label{eq:chain_rule}
\end{equation}
Applying the fundamental theorem of calculus on $[0,1]$:
\begin{equation}
\gamma(x,u) - \gamma(\hat{x},u) = \Phi(1) - \Phi(0) = \int_0^1 \frac{d\Phi}{ds} ds.
\label{eq:integral_representation}
\end{equation}
Since $\tilde{x}$ does not depend on $s$, it can be factored out, yielding \eqref{eq:integral_rep}, which completes the proof. 
\end{proof}
Substituting \eqref{eq:integral_rep} into \eqref{eq:error_dynamics}, the error dynamics become:
\begin{equation}
\dot{\tilde{x}} = \left( A_0 + A_\gamma(t) \right)\tilde{x},
\label{eq:lpv_error}
\end{equation}
where $A_\gamma(t)\in \mathbb{R}^{2n\times 2n}$ is the averaged Jacobian obtained substituting \eqref{eq:chain_rule} in \eqref{eq:integral_representation}:
\begin{equation}
A_\gamma(t) = \int_0^1 \frac{\partial \gamma}{\partial x}(\bar{x}(s), u(t)) \,ds .
\label{eq:averaged_jacobian_def}
\end{equation}
\subsection{Polytopic Embedding}
To enable LMI-based synthesis, we derive a polytopic representation of the averaged Jacobian $A_\gamma(t)$ defined in \eqref{eq:averaged_jacobian_def}. 
\begin{assumption}\label{ass:domain}
The true and estimated states evolve in a compact operating domain 
\begin{equation*}
\mathcal{D} = \prod_{i=1}^{n}[q_{\min}, q_{\max}] \times \prod_{i=1}^{n}[p_{\min}, p_{\max}]
\end{equation*}
and the input satisfies $u(t) \in [u_{\min}, u_{\max}]$ for all $t \geq 0$, with $q_{\min}, q_{\max}, p_{\min}, p_{\max} \in \mathbb{R}^n$ and $u_{\min}, u_{\max} \in \mathbb{R}^m$.
\end{assumption}
Under this assumption, the Jacobian of the nonlinear term $\gamma(\bar{x},u)$ in the state $x$ is bounded and admits the structure:
\begin{equation}
\frac{\partial \gamma}{\partial x}(\bar{x}, u) =
\begin{bmatrix}
0_{n \times n} & 0_{n \times n} \\
\sum_{j=1}^m a^{(j)}(\bar{q}) u_j & 0_{n \times n}
\end{bmatrix}
-
L
\begin{bmatrix}
\beta(\bar{q}, \bar{p}) & \Gamma(\bar{q})
\end{bmatrix},
\label{eq:jacobian_general}
\end{equation}
where the component matrices are defined as:
\begin{align}
a^{(j)}(q) &:= \frac{\partial g_2^{(j)}}{\partial q}(q) \in \mathbb{R}^{n\times n}, \label{eq:aj_def}\\
\beta(q,p) &:= 
\begin{bmatrix}
[a^{(1)}(q)]^T M^{-1}p \\
\vdots \\
[a^{(m)}(q)]^T M^{-1}p
\end{bmatrix} \in \mathbb{R}^{m\times n}, \label{eq:beta_def}\\
\Gamma(q) &:= g_2(q)^T M^{-1} \in \mathbb{R}^{m\times n}. \label{eq:Gamma_def}
\end{align}
Here, $g_2^{(j)}(q) \in \mathbb{R}^n$ denotes the $j$-th column of the input matrix $g_2(q)$, and $a^{(j)}(q)$ is the Jacobian of this column relative to $q$. Each row of $\beta(q,p)$ is a row vector $[a^{(j)}(q)]^T M^{-1}p$.
\begin{remark}
The formulation \eqref{eq:jacobian_general} follows dimensional consistency.
The term $\sum_{j=1}^m a^{(j)}(q)u_j \in \mathbb{R}^{n \times n}$ is a weighted sum of $m$ Jacobian matrices, where the weights are the input components $u_j$.
\end{remark}
By the integral mean value theorem applied to \eqref{eq:averaged_jacobian_def}, there exist intermediate values along the line segment connecting $\hat{x}$ and $x$ such that:
\begin{equation}
A_\gamma(t) = 
\begin{bmatrix}
0_{n \times n} & 0_{n \times n} \\
\sum_{j=1}^m a^{(j)*}(t) u_j(t) & 0_{n \times n}
\end{bmatrix}
-
L
\begin{bmatrix}
\beta^*(t) & \Gamma^*(t)
\end{bmatrix},
\label{eq:averaged_jacobian_general}
\end{equation}
where $a^{(j)*}(t)$, $\beta^*(t)$, and $\Gamma^*(t)$ are bounded element-wise due to Assumption~\ref{ass:domain}.
Moreover, each bounded parameter $\{a,\beta, \Gamma\}\supset\hat{\theta} \in [\theta_{\min}, \theta_{\max}]$ admits a convex decomposition \cite{Ichalal2012}:
\begin{equation}
\hat{\theta}(t) = (1-\mu_{\theta})\theta_{\min} + \mu_{\theta}\theta_{\max}, \quad \mu_{\theta} \in [0,1].
\label{eq:convex_decomp}
\end{equation}
\begin{proposition}[Construction of weighting functions]\label{prop:weighting_functions}
Let $\{\hat{\theta}_1, \ldots, \hat{\theta}_{n_k}\}$ denote the $n_k$ independent bounded parameters in $A_\gamma(t)$ evaluated at $(\hat{x}, u)$, each admitting the convex decomposition \eqref{eq:convex_decomp}. Define the normalized sector variable for parameter $\hat{\theta}_j$ as:
\begin{equation}
\mu_{\hat{\theta}_j} = \frac{\hat{\theta}_j - \theta_{j,\min}}{\theta_{j,\max} - \theta_{j,\min}}, \quad j = 1, \ldots, n_k,
\label{eq:sector_variables}
\end{equation}
and the corresponding elementary weights:
\begin{equation}
w_{\hat{\theta}_j}^L = 1 - \mu_{\hat{\theta}_j}, \quad w_{\hat{\theta}_j}^H = \mu_{\hat{\theta}_j}.
\label{eq:elementary_weights}
\end{equation}
For each vertex $i = (i_1, \ldots, i_{n_k}) \in \{L, H\}^{n_k}$, the weighting function is:
\begin{equation}
h_i(\hat{x}, u) = \prod_{j=1}^{n_k} w_{\hat{\theta}_j}^{i_j},
\label{eq:weighting_function}
\end{equation}
where $i_j \in \{L, H\}$ selects the low or high elementary weight for parameter $\hat{\theta}_j$. By construction, $h_i(\hat{x}, u) \geq 0$ and $\sum_{i=1}^{N_v} h_i(\hat{x}, u) = 1$ for all $(\hat{x}, u)$ with $N_v=2^{n_k}$.
\end{proposition}
\begin{proof}
 Each bounded parameter admits the convex representation $\hat{\theta}_j = w_{\hat{\theta}_j}^L \theta_{j,\min} + w_{\hat{\theta}_j}^H \theta_{j,\max}$, where $w_{\hat{\theta}_j}^L + w_{\hat{\theta}_j}^H = 1$ and both weights are non-negative. The averaged Jacobian $A_\gamma(t)$ is affine in each parameter $\hat{\theta}_j$, hence it can be expressed as a convex combination of $N_v$ vertex matrices $A_i(L)$ corresponding to all combinations of extreme values $(\theta_{1,\min}$ or $\theta_{1,\max}, \ldots, \theta_{n_k,\min}$ or $\theta_{n_k,\max})$. 
The number of vertices depends on the system structure:
\begin{equation}
N_v = 2^{n_k}, \quad \text{where} \quad n_k = n_{nl} + n_u,
\label{eq:num_vertices}
\end{equation}
with $n_{nl}$ denoting the number of independent nonlinear parameters in $A_\gamma(t)$ and $n_u$ the number of independent inputs. The value of $n_k$ depends critically on the structure of system nonlinearities. The weighting function \eqref{eq:weighting_function} is the product of convex weights, hence $h_i(\hat{x}, u) \geq 0$.
\end{proof}
Taking all combinations of extreme values of the bounded parameters yields $N_v$ vertex matrices $A_i(L)$, leading to the polytopic LPV representation:
\begin{equation}
\dot{\tilde{x}} = \sum_{i=1}^{N_v} h_i(\hat{x},u) A_i(L)\tilde{x},
\label{eq:lpv_final}
\end{equation}
where $h_i(\hat{x},u) \geq 0$ and $\sum_{i=1}^{N_v} h_i(\hat{x},u) = 1$ are the weighting functions constructed via Proposition~\ref{prop:weighting_functions}. 

\subsection{Observer Synthesis via LMI}
Each vertex matrix in the polytopic representation \eqref{eq:lpv_final} admits the affine decomposition:
\begin{equation}
A_i(L) = \bar{A}_i - L\bar{C}_i, \quad i = 1, \ldots, N_v,
\label{eq:vertex_decomposition}
\end{equation}
where $\mathbb{R}^{2n \times 2n}\ni\bar{A}_i  = A_0 + \begin{bmatrix}
    0 & 0 \\a_iu_i & 0
\end{bmatrix}$ and $\mathbb{R}^{m \times 2n}\ni\bar{C}_i = \begin{bmatrix}
    \beta_i & \Gamma_i
\end{bmatrix}$ represents the $i$-th vertex of the polytopic system in the form suitable for observer design.
\begin{theorem}\label{th:lmi_synthesis}
Consider the polytopic error dynamics \eqref{eq:lpv_final} with vertex matrices $\{A_i(L)\}_{i=1}^{N_v}$ defined via \eqref{eq:vertex_decomposition}. If there exist a symmetric positive definite matrix $P = P^\top > 0$ and a matrix $\mathcal{K} \in \mathbb{R}^{2n \times m}$ such that:
\begin{equation}
\bar{A}_i^\top P + P\bar{A}_i - \bar{C}_i^\top \mathcal{K} ^\top - \mathcal{K} \bar{C}_i + 2\lambda P < 0,
\label{eq:lmi}
\end{equation}
for all $i = 1, \ldots, N_v$ and some prescribed $\lambda > 0$, then the observer gain $L = P^{-1}\mathcal{K}$ ensures exponential convergence of the estimation error:
\begin{equation}
\|\tilde{x}(t)\|_2 \leq \kappa e^{-\lambda t}\|\tilde{x}(0)\|_2, \quad \forall t \geq 0,
\label{eq:exponential_bound}
\end{equation}
where $\kappa = \sqrt{\lambda_{\max}(P)/\lambda_{\min}(P)}$ is the condition number of $P$.
\end{theorem}
\begin{proof}
Consider the quadratic Lyapunov function candidate $V(\tilde{x}) = \tilde{x}^\top P\tilde{x}$ with $P \succ 0$. Computing its time derivative along the trajectories of \eqref{eq:lpv_final}:
\begin{align*}
\dot{V}(\tilde{x}) &= \tilde{x}^\top P\dot{\tilde{x}} + \dot{\tilde{x}}^\top P\tilde{x} \\
&= \tilde{x}^\top P\left(\sum_{i=1}^{N_v} h_i A_i(L)\tilde{x}\right) + \left(\sum_{i=1}^{N_v} h_i A_i(L)\tilde{x}\right)^\top P\tilde{x} \\
&= \sum_{i=1}^{N_v} h_i \tilde{x}^\top(A_i(L)^\top P + P A_i(L))\tilde{x}.
\end{align*}
By the change of variables $\mathcal{K}  = PL$, condition \eqref{eq:lmi} is equivalent to:
\begin{equation*}
A_i(L)^\top P + P A_i(L) + 2\lambda P < 0, \quad \forall i = 1, \ldots, N_v.
\end{equation*}
Since $h_i(\hat{x},u) \geq 0$ and $\sum_{i=1}^{N_v} h_i(\hat{x},u) = 1$ for all $(\hat{x},u)$:
\begin{equation*}
\dot{V}(\tilde{x}) < -2\lambda V(\tilde{x}), \quad \forall \tilde{x} \neq 0.
\end{equation*}
This implies $V(\tilde{x}(t)) \leq e^{-2\lambda t}V(\tilde{x}(0))$. Using the bounds $\lambda_{\min}(P)\|\tilde{x}\|_2^2 \leq V(\tilde{x}) \leq \lambda_{\max}(P)\|\tilde{x}\|_2^2$, we obtain:
\begin{equation*}
\lambda_{\min}(P)\|\tilde{x}(t)\|_2^2 \leq e^{-2\lambda t}\lambda_{\max}(P)\|\tilde{x}(0)\|_2^2,
\end{equation*}
which yields \eqref{eq:exponential_bound} with $\kappa = \sqrt{\lambda_{\max}(P)/\lambda_{\min}(P)}$. 
\end{proof}
\subsection{Gain-Scheduled Observer}
If the LMI problem \eqref{eq:lmi} admits distinct solutions $\{(P, \mathcal{K} _i)\}_{i=1}^{N_v}$ at each vertex, a gain-scheduled observer can be constructed:
\begin{equation}
L(\hat{x},u) = \sum_{i=1}^{N_v} h_i(\hat{x},u)L_i, \quad L_i = P^{-1}\mathcal{K} _i,
\label{eq:gain_scheduled}
\end{equation}
where the same weighting functions $h_i(\hat{x},u)$ used in the polytopic representation \eqref{eq:lpv_final} interpolate the vertex gains. This approach reduces conservatism by adapting the output injection matrix to the instantaneous operating point, without enforcing a single gain across the entire polytopic domain. The performance advantages of gain scheduling are demonstrated in Section~\ref{sec:Simulation}.

\section{Application to a Mechanical System}\label{sec:Simulation}
\subsection{Simplified model of DEA actuator}
We apply the observer design methodology developed in Section~\ref{sec:Observer design} to a mechanical system with electrostatic actuation. The system is a simplified model of a dielectric elastomer actuator (DEA) \cite{Hammoud2024}, where electrical dynamics are treated as quasi-static, resulting in a second-order mechanical model with state-dependent input mapping.
The state vector is $x = [q, p]^\top \in \mathbb{R}^2$, where $q$ denotes the actuator displacement and $p$ represents the momentum. The input is defined as $u(t) = U(t)^2$, with $U$ being the applied voltage. The Port-Hamiltonian model follows the form \eqref{eq:plant}:
\begin{equation}
    \begin{bmatrix}
        \dot{q}\\\dot{p}
    \end{bmatrix} = \begin{bmatrix}
        0 & 1\\
        -1 & -\eta
    \end{bmatrix}\begin{bmatrix}
        kq\\
        \frac{p}{m}
    \end{bmatrix} + \begin{bmatrix}
        0\\ g_2(q)
    \end{bmatrix} u,
    \quad
    y = \begin{bmatrix}
        0 & g_2(q)
    \end{bmatrix}\begin{bmatrix}
        kq\\
        \frac{p}{m}
    \end{bmatrix},
    \label{eq:ph_scalar_system}
\end{equation}
where the input mapping is:
\begin{equation}
g_2(q) = 2\varepsilon(q + q_0)^3.
\label{eq:g2_electrostatic}
\end{equation}
The system parameters are: mass $m = 1$ kg, stiffness $k = 1000$ N/m, damping $\eta = 50$ N·s/m, reference displacement $q_0 = 10^{-3}$ m, and electrostatic constant $\varepsilon = 2.8$ F/m.

From \eqref{eq:jacobian_general}, the Jacobian structure is now a $2 \times 2$ matrix and reads:
\begin{equation}
\frac{\partial \gamma}{\partial x}(\bar{x}, u) =
\begin{bmatrix}
0 & 0 \\
a(\bar{q})u & 0
\end{bmatrix}
-
L
\begin{bmatrix}
\beta(\bar{q},\bar{p}) & g_2(\bar{q})M^{-1}
\end{bmatrix}, 
\label{eq:jacobian_scalar}
\end{equation}
where the observer gain is $L = [L_1, L_2]^\top \in \mathbb{R}^2$, and the scalar quantities are defined following \eqref{eq:aj_def}--\eqref{eq:Gamma_def}:
\begin{align}
a(q) &= 6\varepsilon(q+q_0)^2, 
\label{eq:a_scalar}\\
\beta(q,p) &= \frac{a(q)p}{m}. 
\label{eq:beta_scalar}
\end{align}
\subsection{Operating Domain and Polytopic Embedding}
The operating domain is determined by open-loop simulation with a step input at $t=1$ s. The maximum stable input amplitude is $U_{\max}=5.14$ kV (thus $\bar{u}=U_{\max}^2 = 26.42$ kV$^2$). This value defines the upper bound for the polytopic embedding under Assumption~\ref{ass:domain}. Note that this bound does not restrict the physical system's admissible input range, but only defines the stability region for the LPV representation.
Under Assumption~\ref{ass:domain}, the operating domain $\mathcal{D}$ for this system is:
\begin{align}
q_{\min} &= -8.1257 \times 10^{-6} \text{ m}, \quad q_{\max} = 4.6754 \times 10^{-4} \text{ m}, \nonumber\\
p_{\min} &= -6.302 \times 10^{-3} \text{ kg·m/s}, \quad p_{\max} = 2.228 \times 10^{-3} \text{ kg·m/s}.\nonumber
\end{align}
For system \eqref{eq:ph_scalar_system}, the polytopic embedding involves four independent bounded parameters evaluated at the estimated state: $\hat{\theta}_1 = \hat{a} = a(\hat{q})$, $\hat{\theta}_2 = u$, $\hat{\theta}_3 = \hat{\beta} = \hat{a}\hat{p}/m$, and $\hat{\theta}_4 = \hat{g} = g_2(\hat{q})$. Following \eqref{eq:num_vertices}, we have $n_{nl} = 3$ (three nonlinear parameters: $a$, $\beta$, $g$) and $n_u = 1$ (single input), yielding $n_k = 4$ and thus $N_v = 2^4 = 16$ vertices.
The vertex matrices $A_i(L)$ in \eqref{eq:lpv_final} are obtained by taking all $16$ combinations of extreme values:
\begin{equation}
(a_i, u_i, \beta_i, g_i) \in \{a_{\min}, a_{\max}\} \times \{0, \bar{u}\} \times \{\beta_{\min}, \beta_{\max}\} \times \{g_{\min}, g_{\max}\},
\end{equation}
where the parameter bounds, computed from the operating domain $\mathcal{D}$ and equations \eqref{eq:a_scalar}--\eqref{eq:beta_scalar}, are:
\begin{align}
a_{\min} &= 1.65281 \times 10^{-5}, & a_{\max} &= 3.61820 \times 10^{-5}, \nonumber\\
\beta_{\min} &= -2.280516 \times 10^{-7}, & \beta_{\max}& = 8.064450 \times 10^{-8},\nonumber\\
g_{\min} &= 5.46459 \times 10^{-9}, & g_{\max} &= 1.76996 \times 10^{-8}. \nonumber
\end{align}
The weighting functions $h_i(\hat{x}, u)$ for $i = 1, \ldots, 16$ are constructed via Proposition~\ref{prop:weighting_functions} using the normalized sector variables and elementary weights for the four parameters.
The initial conditions for the plant and observer are:
\begin{align*}
\text{Plant:} \quad &[q_0, p_0]^\top = [0, 0]^\top, \\
\text{Observer:} \quad &[\hat{q}_0, \hat{p}_0]^\top = [2 \times 10^{-4}, -2 \times 10^{-3}]^\top.
\end{align*}

\subsection{Observer Synthesis and Performance Comparison}
Two observer designs are evaluated: a constant-gain observer and a gain-scheduled observer, both obtained by solving the LMI problem \eqref{eq:lmi} from Theorem~\ref{th:lmi_synthesis}.
The parameter $\lambda$ in \eqref{eq:lmi} represents a \textit{certified worst-case exponential decay rate} from Lyapunov stability theory. Satisfaction of \eqref{eq:lmi} guarantees that the estimation error satisfies the exponential bound \eqref{eq:exponential_bound}:
\begin{equation*}
\|\tilde{x}(t)\|_2 \leq \kappa e^{-\lambda t} \|\tilde{x}(0)\|_2, \quad 
\kappa = \sqrt{\lambda_{\max}(P)/\lambda_{\min}(P)},
\end{equation*}
for all trajectories within the operating domain $\mathcal{D}$. Since the LMI conditions must hold \textit{simultaneously} for all 16 vertices of the polytopic embedding, this introduces conservatism. Consequently, $\lambda$ should be interpreted as a guaranteed lower bound on the exponential decay rate, not necessarily the actual convergence speed observed in simulation.
Table~\ref{table:comparison} shows that the gain-scheduled observer \eqref{eq:gain_scheduled} satisfies the LMI conditions for significantly larger values of $\lambda$, achieving approximately a fivefold improvement in the maximum certifiable decay rate ($\lambda_{\max}^{\text{sched}} = 4.554$ vs. $\lambda_{\max}^{\text{const}} = 0.897$).
\begin{table}[!htpb]
\caption{Maximum Certifiable Decay Rates}
\label{table:comparison}
\centering
\begin{tabular}{|l|c|}
\hline
\textbf{Observer} & $\boldsymbol{\lambda_{\max}}$ (certified) \\
\hline
Constant-Gain  & 0.897 \\
Gain-Scheduled & 4.554 (5.07$\times$ improvement) \\
\hline
\end{tabular}
\vspace{0.5em}
\\ 
\footnotesize
The gain-scheduled design certifies stability at decay rates $5$ times larger than the 
constant-gain approach, reflecting reduced conservatism.
\end{table}

\subsection{Simulation Scenarios}
Two scenarios are considered to compare the constant-gain observer (obtained from \eqref{eq:lmi}) and the gain-scheduled observer \eqref{eq:gain_scheduled}.
\subsubsection{\textbf{Scenario 1 (Identical $\lambda$)}}
Both observers are designed with $\lambda = 0.0897$ (10\% of $\lambda_{\max}^{\text{const}} = 0.897$), isolating the effect of the parameter-dependent structure. As shown in Figs.~\ref{fig:group1} and Table~\ref{table:performance_group1}, both observers exhibit comparable performance in terms of settling time and RMS error. The gain-scheduled observer shows a slightly higher peak momentum error (2.88 vs. 2.46~g$\cdot$m/s). This confirms that a larger maximum certifiable decay rate $\lambda_{\max}$ does not necessarily imply faster convergence when the same $\lambda$ is used.
\begin{figure}[!t]
    \centering
    \includegraphics[width=0.5\textwidth]{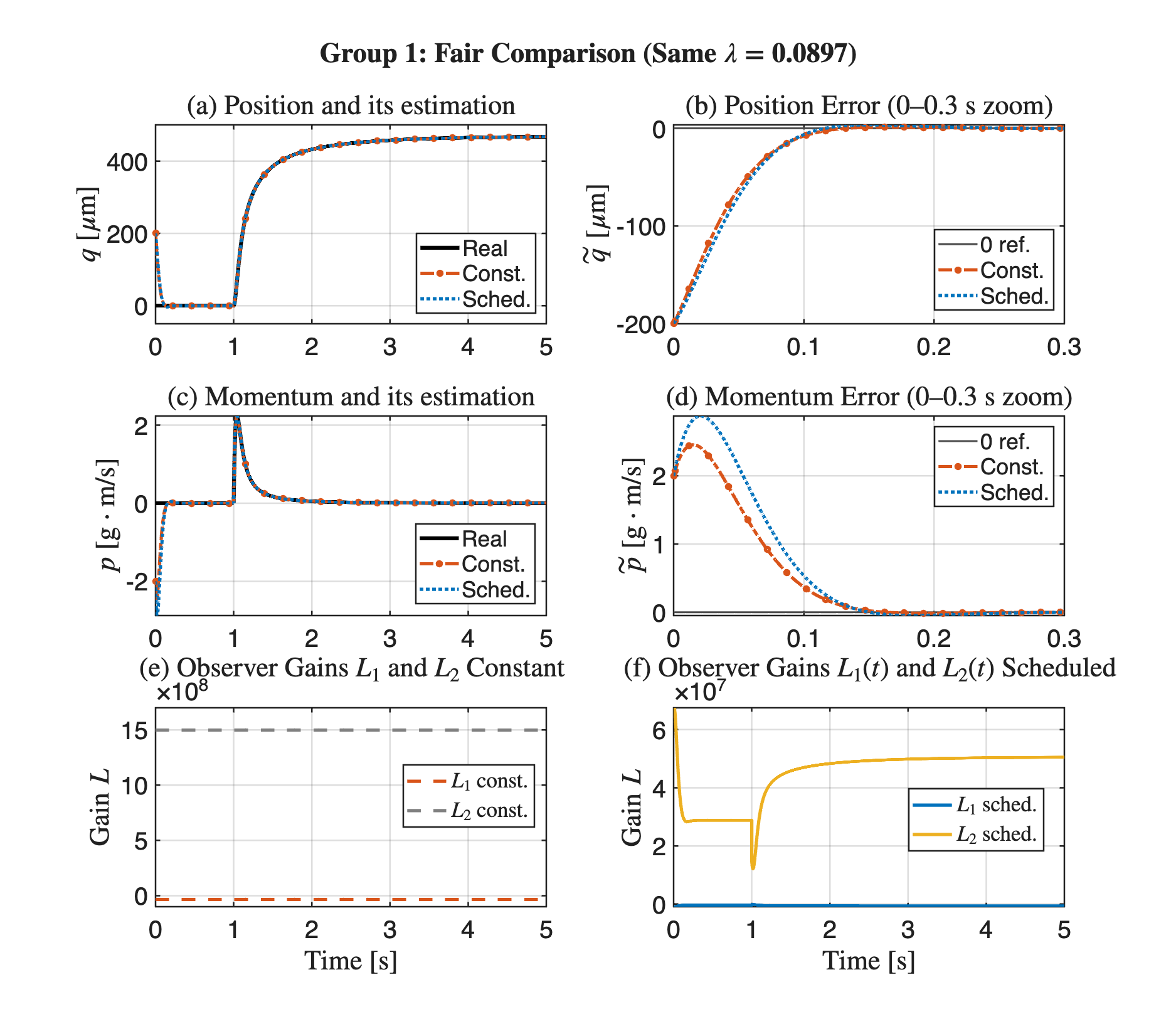}
    \caption{Scenario 1 (Fair comparison, $\lambda = 0.0897$): 
    (a) Position estimation $q$ vs. $\hat{q}$. 
    (b) Position error $\tilde{q}$ (0--0.4~s zoom). 
    (c) Momentum estimation $p$ vs. $\hat{p}$. 
    (d) Momentum error $\tilde{p}$ (0--0.4~s zoom). 
    (e) Observer gains $L_1(t)$ and $L_2(t)$.
    (f) Observer scheduled gains $L_1(t)$ and $L_2(t)$.
    }
    \label{fig:group1}
\end{figure}
\begin{table}[!t]
\caption{Performance Comparison --- Scenario 1 (Fair Comparison, $\lambda = 0.0897$)}
\label{table:performance_group1}
\centering
\begin{tabular}{|l|c|c|}
\hline
\textbf{Metric} & \textbf{Constant-Gain} & \textbf{Gain-Scheduled} \\
\hline
Peak $|\tilde{q}|$ [$\mu$m] & 200 & 200 \\
Peak $|\tilde{p}|$ [g$\cdot$m/s] & 2.46 & 2.88 \\
Peak $\|\tilde{x}\|$ [g$\cdot$m/s] & 2.46 & 2.88 \\
RMS $\|\tilde{x}\|$ [g$\cdot$m/s] & 0.240 & 0.295 \\
Settling time $T_s$ (2\%) [s] & 0.141 & 0.139 \\
Overshoot $|\tilde{p}|$ [\%] & 22.8 & 43.8 \\
\hline
\end{tabular}
\vspace{0.5em}
\\ 
\footnotesize
At identical $\lambda = 0.0897$, both observers exhibit comparable performance, 
confirming that larger $\lambda_{\max}$ does not guarantee faster convergence.
\end{table}
\subsubsection{\textbf{Scenario 2 (Exploiting scheduling)}}
The constant-gain observer is evaluated at its maximum feasible value $\lambda_{\max}^{\text{const}} = 0.897$, while the gain-scheduled observer is tested at $\lambda = 0.897$ and at its maximum feasible value $\lambda = 4.554$, for which the constant-gain LMI conditions are infeasible. At $\lambda = 4.554$, the gain-scheduled observer achieves a 31\% reduction in peak momentum error and a 34\% reduction in RMS error compared to the constant-gain design at $\lambda = 0.897$ (see Table~\ref{table:performance_group2} and Fig.~\ref{fig:group2}). Moreover, overshoot is eliminated (0\% vs. 45\%), at the cost of a longer settling time.
\textbf{Key insight:} the gain-scheduled observer enables certified stability in regimes where the constant-gain design becomes infeasible. The main advantage of scheduling is therefore reduced conservatism, rather than improved performance at identical tuning levels. The time evolution of the scheduled gains is reported in Figs.~\ref{fig:group1}(e)--(f) and Fig.~\ref{fig:sched_group2}.
\begin{figure}[!t]
    \centering
    \includegraphics[width=0.5\textwidth]{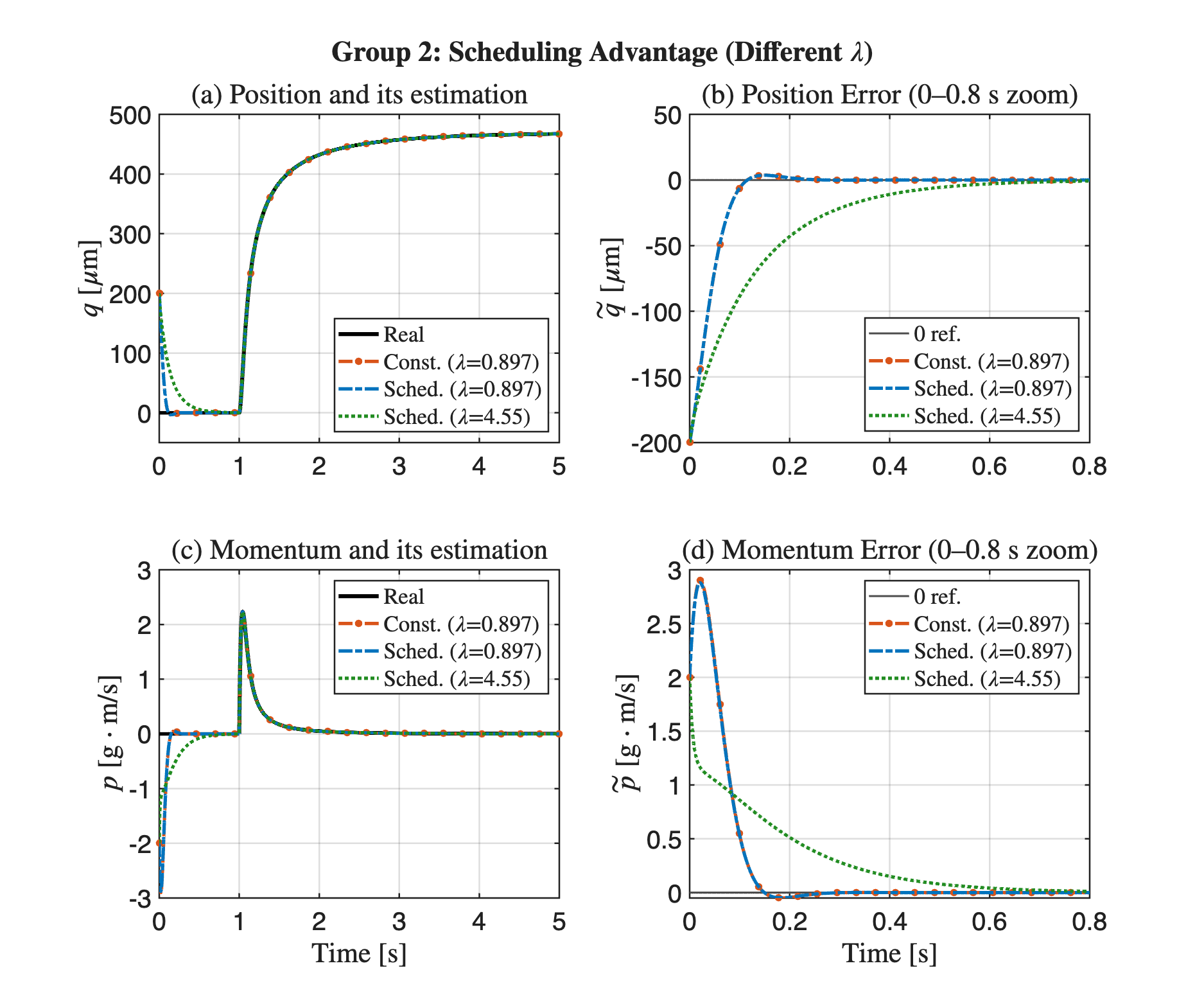}
    \caption{Scenario 2 (Different $\lambda$ values): 
    (a) Position estimation. 
    (b) Position error (0--0.8~s zoom). 
    (c) Momentum estimation. 
    (d) Momentum error (0--0.8~s zoom).}
    \label{fig:group2}
\end{figure}
\begin{figure}[!t]
    \centering
    \includegraphics[width=0.31\textwidth]{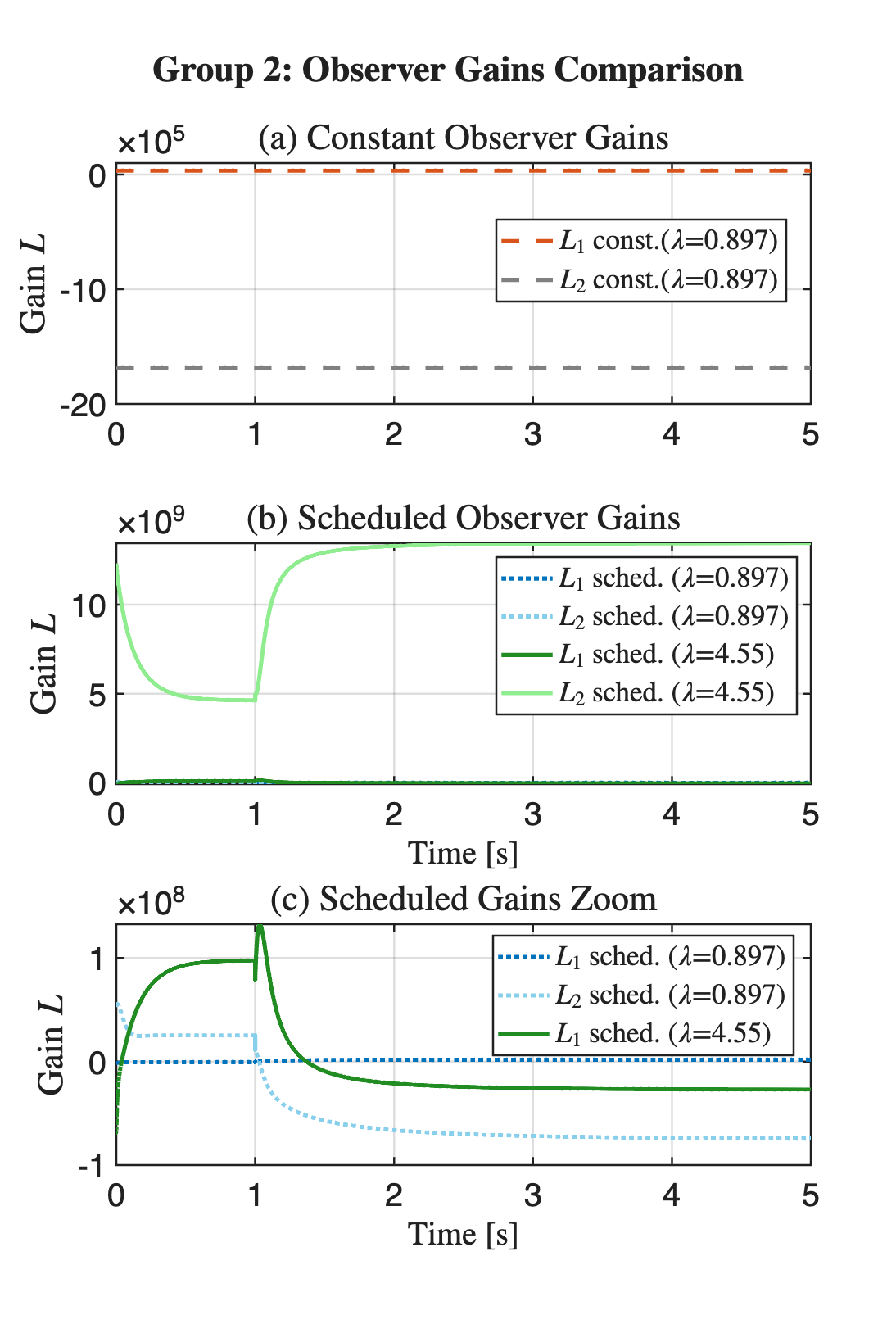}
    \caption{(a) Constant Observer gains, only for $\lambda = 0.897$ since for $\lambda = 4.55$ the constant $L$ approach fails.
    (b) Scheduled Observer gains.
    (c) Scheduled Observer gains zoom. 
    }
    \label{fig:sched_group2}
\end{figure}
\begin{table}[!t]
\caption{Performance Comparison --- Scenario 2 (Different $\lambda$ Values)}
\label{table:performance_group2}
\centering
\begin{tabular}{|l|c|c|c|}
\hline
\textbf{Metric} & \textbf{Const} & \textbf{Sched} & \textbf{Sched} \\
 & ($\lambda$=0.897) & ($\lambda$=0.897) & ($\lambda$=4.554) \\
\hline
Peak $|\tilde{q}|$ [$\mu$m] & 200 & 200 & 200 \\
Peak $|\tilde{p}|$ [g$\cdot$m/s] & 2.90 & 2.88 & 2.00 \\
Peak $\|\tilde{x}\|$ [g$\cdot$m/s] & 2.90 & 2.88 & 2.01 \\
RMS $\|\tilde{x}\|$ [g$\cdot$m/s] & 0.298 & 0.295 & 0.198 \\
Settling time $T_s$ (2\%) [s] & 0.138 & 0.138 & 0.607 \\
Overshoot $|\tilde{p}|$ [\%] & 45.0 & 44.0 & 0.0 \\
\hline
\multicolumn{4}{|l|}{\footnotesize Improvement (Sched $\lambda$=4.554 vs. Const): 
Peak $\|\tilde{x}\|$: +31\%, RMS: +34\%} \\
\hline
\end{tabular}
\vspace{0.5em}
\\ 
\footnotesize
The scheduled observer at $\lambda = 4.554$ (infeasible for constant-gain) 
achieves substantial overshoot reduction at the cost of longer settling time.
\end{table}


The simulation results (Tables~\ref{table:performance_group1}--\ref{table:performance_group2}) show that when both observers are designed with the same $\lambda$, their performance is comparable. This indicates that the larger $\lambda_{\max}$ achieved by the gain-scheduled design reflects reduced conservatism in the synthesis conditions rather than inherent performance improvement. The advantage of the gain-scheduled observer becomes evident when operating at $\lambda$ values for which the constant-gain LMI conditions are infeasible.

The constant-gain observer is preferable when its LMI conditions are feasible due to its simpler implementation. However, for systems with strong nonlinearities—particularly in the input matrix—the gain-scheduled approach provides a less conservative alternative that enables observer synthesis when the constant-gain design fails. This is especially relevant for pH systems with state-dependent input matrices $g(\cdot)$, where existing methods (e.g., contraction-based observers~\cite{Spirito2024}) rely on restrictive constant input assumptions. The gain-scheduled observer achieves reduced conservatism by adapting the output injection matrix $L(\hat{x},u)$ to the operating point through interpolation of vertex-dependent gains.

\section{Conclusions and Perspectives}\label{sec:conclusions} 
This paper presents an LMI-based observer design methodology for a class of port-Hamiltonian systems with state-dependent input matrices, combining integral mean value representation with polytopic LPV embedding. The approach targets electromechanical systems, where quasi-static electrical assumptions induce nonlinear input maps that violate the constant input matrix requirement of existing contraction-based observer designs. It introduces a systematic polytopic embedding framework that enables tractable observer synthesis for systems with Lipschitz-continuous nonlinear input matrices, and compares constant-gain and gain-scheduled observer designs. The constant-gain observer is preferable when feasible due to its simplicity, whereas the gain-scheduled approach provides a less conservative alternative in highly nonlinear cases, allowing significantly larger certifiable decay rates (e.g., $\lambda_{max}$ from $0.897$ to $4.554$ in the studied numerical example). Simulation results show that higher decay rates reflect reduced conservatism in the synthesis conditions rather than guaranteed performance improvement; both observers perform similarly under identical tuning, with scheduling advantages arising only when operating beyond the feasibility range of constant gains.
The ongoing work has two main directions. First, we aim to extend the proposed method to more complex nonlinear systems, including those with state-dependent interconnection and dissipation matrices and/or non-quadratic Hamiltonian functions. Second, we plan to validate the proposed observer experimentally, for example, on HASEL or DEA actuator platforms \cite{Cisneros2025Mechatronics}.



\end{document}